\documentclass{amsart}

\usepackage{color}

\newtheorem{theorem}{Theorem}[section]
\newtheorem*{theorem A}{Theorem A}
\newtheorem*{theorem B}{N\"olker's Theorem}

\newtheorem{corollary}{Corollary}[section]

\theoremstyle{remark}
\newtheorem{remark}{Remark}[section]
\theoremstyle{remark}

\theoremstyle{definition}

\newcommand{\trace}{\operatorname{trace}}

\newtheorem{example}{Example}[section]
\numberwithin{equation}{section}
\def\({\left ( }
\def\){\right )}
\def\<{\left < }
\def\>{\right >}

\begin{document}

\newcommand{\W}{\mathcal{W}}
\newcommand{\K}{\mathcal{K}}

%\noindent {\sc {International Electronic Journal of Geometry}

%\noindent {\sc \small Volume 1  No. ? pp. 000--000 (2008) \copyright
%IEJG}

\vspace{2cm}

\title[Holomorphic submanifolds of some hypercomplex manifolds]
{Holomorphic submanifolds of some hypercomplex manifolds with Hermitian and Norden metrics}

\author{Galia Nakova}
\address{Department of Algebra and Geometry,
Faculty of Mathematics and Informatics,
University of Veliko Tarnovo "St. Cyril and St. Methodius",
2 T.Tarnovski Str., 5003 Veliko Tarnovo, Bulgaria}
\email{gnakova@gmail.com}

\author{Hristo Manev}
\address{Department of Medical Informatics, Biostatistics and Electronic Education, Faculty of Public Health, Medical University of Plovdiv, 15A Vasil Aprilov Blvd, 4002 Plovdiv, Bulgaria}
\email{hmanev@meduniversity-plovdiv.bg}

%\thanks{The author is supported by ...}


%\date{January 1, 2007 and, in revised form, June 22, 2007.}

%\dedicatory{{\rm (Communicated by H. Hilmi HACISALIHO\v GLU)}}


\begin{abstract}
In this paper we initiate the study of submanifolds of almost hypercomplex manifolds with Hermitian and Norden metrics. Object of investigations are holomorphic
submanifolds of the hypercomplex manifolds which are locally conformally equivalent to the hyper-K\"ahler manifolds of the considered type. Necessary and sufficient conditions the investigated holomorphic submanifolds to be totally umbilical or totally geodesic are obtained. Examples of the examined submanifolds are constructed.
\end{abstract}

\newcommand{\R}{\mathbb{R}}

\maketitle

\section*{Introduction} The almost hypercomplex manifolds with Hermitian and Norden  metrics have been introduced by Gribachev, Manev and Dimiev in
\cite{GMD}. These manifolds are equipped with an almost hypercomplex structure $H=(J_1,J_2,J_3)$ and a metric structure $G=(g,g_1,g_2,g_3)$. Here $g$ is a neutral
metric, which is Hermitian with respect to the almost complex structure $J_1$ of $H$ and $g$ is a Norden metric (known also as an anti-Hermitian metric) regarding the almost complex structures $J_2$ and $J_3$ of $H$. Moreover, $G$ contains the K\"ahler 2-form $g_1$ with respect to $J_1$ and two associated Norden metrics $g_2, g_3$ with respect to $J_2$ and $J_3$, respectively. The geometry of almost hypercomplex manifolds with Hermitian and Norden metrics has been investigated in \cite{GMD, MM, MG, MS}. This type of manifolds are the only possible case to involve Norden metrics on almost hypercomplex manifolds.

In this paper we initiate the study of submanifolds of the considered manifolds. As first step in this direction we study their holomorphic submanifolds. A holomorphic submanifold $M$ of an almost hypercomplex manifold with Hermitian and Norden metrics $(\overline M,H,G)$ is a non-degenerate submanifold such that the tangent bundle
is preserved by the structure $H$, which implies that $M$ is also an almost hypercomplex manifold with Hermitian and Norden metrics with respect to the restrictions of
the structures $H$ and $G$ on $M$.

An almost hypercomplex structure $H$ on $\overline M$ is called hypercomplex if the three almost complex structures $J_1,J_2,J_3$ are integrable. Objects of special interest
in this work are ambient manifolds $(\overline M,H,G)$ belonging to the class denoted by $\W$ of hypercomplex manifolds with Hermitian and Norden metrics. This is the class of the locally conformally equivalent manifolds to the manifolds in the class $\K$ consisting of the hyper-K\"ahler manifolds of the considered type. The class $\K$ is an important subclass of $\W$, where the considered manifolds have parallel $J_1,J_2,J_3$ with respect to the Levi-Civita connection $\nabla$ of $g$.

The present paper is organized as follows. In Section 1 we present some definitions and facts about almost hypercomplex manifolds with Hermitian and Norden metrics.
Section 2 is devoted to the study of holomorphic submanifolds of the considered manifolds belonging to the classes $\W$ and $\K$. We prove that a holomorphic
submanifold $M$ of a $\W$-manifold $\overline M$ is either a $\W$-manifold, or a  $\K$-manifold. Moreover, we obtain necessary and sufficient conditions $M$ to be in $\W$ or $\K$ in terms of the Lee 1-forms $\overline \theta _\alpha $ $(\alpha =1,2,3)$ of $\overline M$. We also show that a holomorphic
submanifold $M$ of a $\W$-manifold $\overline M$ is either totally umbilical, or totally geodesic and find necessary and sufficient conditions for this, expressed by
conditions for $\overline \theta _\alpha $ $(\alpha =1,2,3)$. We obtain that every holomorphic submanifold of a $\K$-manifold is a totally geodesic $\K$-manifold and every holomorphic $\K$-submanifold of a
$\W$-manifold is totally umbilical. In Section 3 we construct examples of the studied submanifolds.
\section{Preliminaries}
A $4n$-dimensional differentiable manifold $(M,H)$ is called {\it an almost hypercomplex manifold} \cite{AM} if it is equipped with {\it an almost hypercomplex structure}
$H=(J_1,J_2,J_3)$, which is a triple of almost complex structures having the properties:
\[
J_\alpha =J_\beta \circ J_\gamma =-J_\gamma \circ J_\beta , \quad J_\alpha ^2=-I
\]
for all cyclic permutations $(\alpha , \beta , \gamma )$ of $(1, 2, 3)$ and the identity $I$.

Let $g$ be a pseudo-Riemannian metric on $(M,H)$ which is Hermitian with respect to  $J_1$ and $g$ is a Norden metric with respect to $J_2$ and
$J_3$, i.e.
\begin{equation}\label{2.1}
g(J_1X,J_1Y)=-g(J_2X,J_2Y)=-g(J_3X,J_3Y)=g(X,Y), \quad X, Y \in TM.
\end{equation}
%Here and further we denote arbitrary vector fields on $M$ by $X$, $Y$, $Z$.
The associated bilinear forms $g_1$, $g_2$ and $g_3$ are determined by
\begin{equation}\label{2.2}
g_1(X,Y)=g(J_1X,Y), \quad g_2(X,Y)=g(J_2X,Y), \quad g_3(X,Y)=g(J_3X,Y).
\end{equation}
According to \eqref{2.1} and \eqref{2.2} the metric $g$ and the associated bilinear forms $g_2$ and $g_3$ are necessarily pseudo-Riemannian metrics of neutral
signature $(2n,2n)$ and $g_1$ is the known K\"ahler 2-form $\Phi $ with respect to $J_1$.

Differentiable $4n$-dimensional  manifolds $M$ equipped with structures $(H,G)$
$=(J_1,J_2,J_3,g,g_1,g_2,g_3)$ are studied in \cite{GMD, MS} and \cite{MM, MG} (under
the name {\it almost hypercomplex pseudo-Hermitian manifolds} and {\it almost hypercomplex manifolds with Hermitian and anti-Hermitian metrics}, respectively). In this paper we refer to $(M,H,G)$ as {\it an almost hypercomplex manifold with Hermitian and Norden metrics}.

The $(0,3)$-tensors $F_\alpha (X,Y,Z)=g((\nabla _XJ_\alpha )Y,Z), \, (\alpha =1,2,3)$, where $\nabla $ is the Levi-Civita connection generated by $g$, are
called {\it fundamental tensors} of $(M,H,G)$. It is well known that the almost hypercomplex structure $H=(J_1,J_2,J_3)$ is a hypercomplex structure if the Nijenhuis
tensor $N_\alpha (X,Y)$ for $J_\alpha $ vanishes for each $\alpha =1,2,3$. Moreover, an almost hypercomplex structure $H$ is hypercomplex if and only if two of the
tensors $N_\alpha $ vanish \cite{AM}.

Let us remark that $(M,H,G)$ is an indefinite almost Hermitian manifold with respect to $J_1$ and it is an almost complex manifold with Norden metric with respect to $J_2$ and $J_3$. The basic classifications of the almost complex manifolds with Hermitian metric
and with Norden metric are given in \cite{GH} and \cite{GB}, respectively.

Let we assume that $(M,H,G)$ belongs to the class ${\W}_4$ from the Gray-Hervella classification, which is a subclass of the class of Hermitian manifolds and it is the class of locally conformal equivalent manifolds to the K\"ahler manifolds. Then the almost complex structure $J_1$ is integrable and $F_1$ is given by
\begin{equation}\label{2.3}
\begin{array}{lc}
F_1(X,Y,Z)=\frac{1}{2(2n-1)}\left[g(X,Y)\theta _1(Z)-g(X,Z)\theta _1(Y)\right. \\
\qquad \qquad \qquad\qquad\qquad\left.-g(X,J_1Y)\theta _1(J_1Z)+g(X,J_1Z)\theta _1(J_1Y)\right] ,
\end{array}
\end{equation}
where the Lee form $\theta _1$ is defined by $\theta _1(Z)=g^{ij}F_1(e_i,e_j,Z)$ for a basis $\left\{e_i\right\}$, $(i=1, \ldots ,4n)$ and $(g^{ij})$ is the inverse matrix of the matrix $(g_{ij})$ of the metric $g$.

One of the basic classes of the integrable almost complex manifolds with Norden metric is ${\W}_1$. It is a subclass of the integrable (almost) complex manifolds with Norden metric and it is the class of the locally conformal equivalent manifolds to the K\"ahler manifolds with Norden metric. If $(M,H,G)$ belongs to ${\W}_1(J_\alpha )$, then
$J_\alpha $ ($\alpha =2,3$) is integrable and the following equality holds
\begin{equation}\label{2.4}
\begin{array}{lc}
F_\alpha (X,Y,Z)=\frac{1}{4n}\left[g(X,Y)\theta _\alpha (Z)+g(X,Z)\theta _\alpha(Y)\right. \\
\qquad \qquad \qquad\qquad\left.+g(X,J_\alpha Y)\theta _\alpha (J_\alpha Z)+g(X,J_\alpha Z)\theta _\alpha (J_\alpha Y)\right] ,
\end{array}
\end{equation}
where the Lee form $\theta _\alpha$ is defined by $\theta _\alpha (Z)=g^{ij}F_\alpha (e_i,e_j,Z)$ ($\alpha =2,3$) for a basis $\left\{e_i\right\}$, $(i=1, \ldots ,4n)$.

The class $\W$, studied in \cite{GMD}, consists of all hypercomplex manifolds with Hermitian and Norden metrics such that $F_1$ satisfies \eqref{2.3} and $F_2, F_3$
satisfy \eqref{2.4}. A manifold $(M,H,G)$ belonging to the class $\W$ is called briefly a {\it $\W$-manifold}.

It is known \cite{GMD} that necessary and sufficient conditions $(M,H,G)$ to be a $\W$-manifold are
\begin{equation}\label{2.5}
\theta _\alpha \circ J_\alpha =-\frac{2n}{2n-1}\theta _1\circ J_1 ,\quad \alpha =2,3.
\end{equation}

According to \cite{GMD}, an almost hypercomplex manifold with Hermitian and Norden metrics is called a hyper-K\"ahler manifold of the considered type if $\nabla {J_\alpha }=0$ ($\alpha =1,2,3$) with
respect to the Levi-Civita connection generated by $g$. The class of these manifolds is denoted by $\K$ in \cite{GMD} and thus we call them \emph{$\K$-manifolds}. It is clear that for the $\K$-manifolds the conditions $F_\alpha =0$
($\alpha =1,2,3$) hold and therefore $\theta _\alpha =0$ ($\alpha =1,2,3$) and $\K$ is a subclass of $\W$.

\section{Holomorphic submanifolds of ${\W}$-manifolds}
%\begin{definition}\label{definition 3.1}
A $4m$-dimensional submanifold $M$ of a $4n$-dimensional $(m<n)$ almost hypercomplex manifold with Hermitian and Norden metrics $(\overline M,H,G)$ is said to be
{\it a holomorphic submanifold} if the tangent bundle $TM$ is preserved by $J_\alpha $ ($\alpha =1,2,3$), i.e. $J_\alpha (T_pM)=T_pM$ for all $p\in M$ and the
restriction of the metric $g$ on $TM$ has a maximal rank.
%\end{definition}

We denote the restrictions of $g$ and $J_\alpha $ ($\alpha =1,2,3$) on $TM$ by the same letters.\\
Let $\overline \nabla $ and $\nabla $ be the Levi-Civita connections of $(\overline M,H,G)$ and its submanifold $M$, respectively. Then the Gauss-Weingarten formulae are given by
\begin{equation}\label{3.1}
\overline \nabla _XY=\nabla _XY+h(X,Y) ,
\end{equation}
\begin{equation}\label{3.2}
\overline \nabla _XN=-A_NX+D_XN ,
\end{equation}
where $X,Y\in TM, \, N\in TM^\bot $, $h$ is the second fundamental form of $M$, $A_N$ is the shape operator in the direction of a normal vector field $N$ and
$D$ is the normal connection. The shape operator and the  second fundamental form are related as usually by $g(A_NX,Y)=g(h(X,Y),N)$.

A submanifold $M$ is said to be {\it totally
umbilical} if $h(X,Y)=g(X,Y)C$, where $C=\frac{1}{4m}\trace h$ is the {\it mean curvature vector} of $M$. If $h(X,Y)=0$, then $M$ is said to be
{\it totally geodesic}.

Let $(\overline M,H,G)$ be a $\W$-manifold. We denote by $\overline F_\alpha $, $\overline \theta _\alpha $ ($\alpha =1,2,3$) the fundamental tensors and the Lee forms of $\overline M$. Then the equalities \eqref{2.3},  \eqref{2.4} and \eqref{2.5} hold. If $p_\alpha $ are the covectors of $\overline \theta _\alpha $,
i.e. $\overline \theta _\alpha (Z)=g(Z,p_\alpha )$ for any $Z\in T\overline M$ and $\alpha =1,2,3$, from \eqref{2.5} we obtain
\begin{equation}\label{3.3}
J_\alpha p_\alpha=\frac{2n}{2n-1}J_1p_1 , \quad \alpha =2,3.
\end{equation}
Let $M$ be a submanifold of $(\overline M,H,G)$. Then for every $p_\alpha $ the following decomposition is valid
\begin{equation}\label{3.4}
p_\alpha =p_\alpha^\top +p_\alpha^\bot,
\end{equation}
where $p_\alpha^\top \in TM$ and $p_\alpha^\bot \in TM^\bot $. Taking into account \eqref{3.4} we conclude that $\overline \theta _\alpha $ ($\alpha =1,2,3$) vanishes on $TM$ (resp. $TM^\bot $) if and only if $p_\alpha^\top$ (resp. $p_\alpha^\bot$) vanishes.
%\begin{proof}
%Taking into account \eqref{3.4},
%for any $Z\in TM$ and any $N\in TM^\bot $ we have
%\begin{equation}\label{3.5}
%\overline \theta _\alpha (Z)=g(Z,p_\alpha^\top ), \quad \overline \theta _\alpha (N)=g(N,p_\alpha^\bot ) , \quad \alpha =1,2,3.
%\end{equation}
%Since $TM$ and $TM^\bot $ are non-degenerate, from \eqref{3.5} it follows the truth of the assertion in the lemma.
%\end{proof}

%\begin{lemma}\label{lemma 3.2}
Let $M$ be a holomorphic submanifold of a $\W$-manifold $(\overline M,H,G)$.
%
%(i) If for some $\alpha \in \{1,2,3\}$  $\overline \theta _\alpha $ vanishes on $TM$ (resp. $TM^\bot $), then all $\overline \theta _\alpha $ vanish on
%$TM$ (resp. $TM^\bot $).
%\par
%(ii) If for some $\alpha \in \{1,2,3\}$ $\overline \theta _\alpha $ is non-zero on $TM$ (resp. $TM^\bot $), then all $\overline \theta _\alpha $ are
%non-zero on $TM$ (resp. $TM^\bot $).
Then from \eqref{2.5} it follows that $\overline \theta _1$, $\overline \theta _2$ and $\overline \theta _3$ are either all non-zero on $TM$ (resp. $TM^\bot $), or all vanish on $TM$ (resp. $TM^\bot $).
%\end{lemma}
%\begin{proof}
%Only the following two cases for $\overline \theta _1$, $\overline \theta _2$ and $\overline \theta _3$ on $TM$ (resp. $TM^\bot $) are possible:
%all $\overline \theta _\alpha $ ($\alpha =1,2,3$) are non-zero on $TM$ (resp. $TM^\bot $);
%there exists $\alpha \in \{1,2,3\}$ such that $\overline \theta _\alpha $ vanishes on $TM$ (resp. $TM^\bot $). We will prove that in the second case all
%$\overline \theta _\alpha $ ($\alpha =1,2,3$) vanish on $TM$ (resp. $TM^\bot $). \\
%Substituting \eqref{3.4} in \eqref{3.3} we have
%\begin{equation}\label{3.6}
%J_\alpha (p_\alpha^\top )+J_\alpha (p_\alpha^\bot )=\frac{2n}{2n-1}(J_1(p^\top_1)+J_1(p^\bot_1)) , \quad \alpha =2,3.
%\end{equation}
%Having in mind that $M$ is a holomorphic submanifold of $\overline M$, for the tangent  and normal part of \eqref{3.6} we obtain
%\begin{equation}\label{3.7}
%J_\alpha (p_\alpha^\top )=\frac{2n}{2n-1}J_1(p^\top_1) \quad {\text {\rm and}} \quad  J_\alpha (p_\alpha^\bot )=\frac{2n}{2n-1}J_1(p^\bot_1) , \quad \alpha =2,3,
%\end{equation}
%respectively. Let we assume that there exists $\alpha \in \{1,2,3\}$ such that $\overline \theta _\alpha $ vanishes on $TM$ (resp. $TM^\bot $). Then
%by using Lemma \ref{lemma 3.1} and \eqref{3.7} we obtain that all $\overline \theta _\alpha $ ($\alpha =1,2,3$) vanish on $TM$ (resp. $TM^\bot $).
%\end{proof}
%\begin{lemma}\label{lemma 3.3}
Moreover, having in mind \eqref{3.4} and the fact that $\overline \theta _\alpha \neq 0$  ($\alpha =1,2,3$)  for a $\W$-manifold which is not a $\K$-manifold, we establish that if all $\overline \theta _\alpha $ ($\alpha =1,2,3$) vanish on
$TM$ (resp. $TM^\bot $), then all $\overline \theta _\alpha $ ($\alpha =1,2,3$) are non-zero on $TM^\bot $ (resp. $TM$).
%If for some $\alpha \in \{1,2,3\}$ $\overline \theta _\alpha $ vanishes on
%$TM$ (resp. $TM^\bot $), then $\overline \theta _\alpha $ is non-zero on $TM^\bot $ (resp. $TM$).
%\end{lemma}
%\begin{proof}
%Let we assume that all $\overline \theta _\alpha $ ($\alpha =1,2,3$) vanish on both $TM$ and $TM^\bot $. Then it follows that all $\overline \theta _\alpha $
%($\alpha =1,2,3$) vanish on $T\overline M$. This is a contradiction with $\overline M$ is a $\W$-manifold, which completes the proof.
%\end{proof}

\begin{theorem}
Let $M$ be a holomorphic submanifold of a $\W$-manifold $(\overline M,H,G)$. Then we get
\begin{equation}\label{3.8}
h(X,Y)=\frac{1}{2(2n-1)}g(X,Y)J_1(p^\bot_1)=\frac{1}{4n}g(X,Y)J_\alpha (p_\alpha^\bot ), \quad \alpha =2,3 ,
\end{equation}
\begin{equation}\label{3.9}
\begin{array}{ll}
(\nabla _XJ_1)Y=\frac{1}{2(2n-1)}\left[g(X,Y)p^\top_1-\overline \theta _1(Y)X+g(X,J_1Y)J_1(p^\top_1)\right. \\ \\
\qquad \qquad \qquad\qquad\quad\left.-\overline \theta _1(J_1Y)J_1X\right] ,
\end{array}
\end{equation}
\begin{equation}\label{3.10}
\begin{array}{ll}
(\nabla _XJ_\alpha )Y=\frac{1}{4n}\left[g(X,Y)p_\alpha^\top +\overline \theta _\alpha (Y)X+g(X,J_\alpha Y)J_\alpha (p_\alpha^\top )\right. \\ \\
\qquad \qquad \qquad\quad\left.+\overline \theta _\alpha (J_\alpha Y)J_\alpha X\right] , \quad \alpha =2,3 ,
\end{array}
\end{equation}
\begin{equation}\label{3.11}
A_{J_1N}X=J_1(A_NX)+\frac{1}{2(2n-1)}\left[\overline \theta _1(N)X+\overline \theta _1(J_1N)J_1X\right] ,
\end{equation}
\begin{equation}\label{3.12}
A_{J_\alpha N}X=J_\alpha (A_NX)-\frac{1}{4n}\left[\overline \theta _\alpha (N)X+\overline \theta _\alpha (J_\alpha N)J_\alpha X\right] , \quad \alpha =2,3 ,
\end{equation}
\begin{equation}\label{3.13}
D_XJ_\alpha N=J_\alpha (D_XN) , \quad \alpha =1,2,3 ,
\end{equation}
where $X,Y\in TM$ and $N\in TM^\bot $.
\end{theorem}
\begin{proof}
Using \eqref{3.1}, we obtain
\begin{equation}\label{3.14}
(\overline \nabla _XJ_\alpha )Y=(\nabla _XJ_\alpha )Y+h(X,J_\alpha Y)-J_\alpha h(X,Y) , \quad \alpha =1,2,3 .
\end{equation}
From \eqref{2.3} and \eqref{2.4} for $\overline F_1$ and $\overline F_\alpha $ ($\alpha =2,3$), respectively, we get
\begin{equation}\label{3.15}
\begin{array}{ll}
(\overline \nabla _XJ_1)Y=\frac{1}{2(2n-1)}\left[g(X,Y)p_1-\overline \theta _1(Y)X+g(X,J_1Y)J_1p_1\right. \\
\qquad \qquad \qquad\qquad\quad\left.-\overline \theta _1(J_1Y)J_1X\right] ,
\end{array}
\end{equation}
\begin{equation}\label{3.16}
\begin{array}{ll}
(\overline \nabla _XJ_\alpha )Y=\frac{1}{4n}\left[g(X,Y)p_\alpha +\overline \theta _\alpha (Y)X+g(X,J_\alpha Y)J_\alpha p_\alpha \right. \\
\qquad \qquad \qquad\quad\left.+\overline \theta _\alpha (J_\alpha Y)J_\alpha X\right] .
\end{array}
\end{equation}
By substituting \eqref{3.15} in \eqref{3.14} and taking into account that $M$ is holomorphic, we obtain \eqref{3.9} and
\begin{equation}\label{3.17}
h(X,J_1Y)-J_1h(X,Y)=\frac{1}{2(2n-1)}\left[g(X,Y)p^\bot_1+g(X,J_1Y)J_1(p^\bot_1)\right].
\end{equation}
Analogously, using \eqref{3.14} and \eqref{3.16}, we get \eqref{3.10}  and
\begin{equation}\label{3.18}
h(X,J_\alpha Y)-J_\alpha h(X,Y)=\frac{1}{4n}\left[g(X,Y)p_\alpha^\bot +g(X,J_\alpha Y)J_\alpha (p_\alpha^\bot )\right] , \ \alpha =2,3 .
\end{equation}
We replace $X,Y$ by $Y,X$ in \eqref{3.17} and \eqref{3.18} and we find that
\begin{equation}\label{3.19}
h(Y,J_1X)-J_1h(X,Y)=\frac{1}{2(2n-1)}\left[g(X,Y)p^\bot_1-g(X,J_1Y)J_1(p^\bot_1)\right] ,
\end{equation}
\begin{equation}\label{3.20}
h(Y,J_\alpha X)-J_\alpha h(X,Y)=\frac{1}{4n}\left[g(X,Y)p_\alpha^\bot +g(X,J_\alpha Y)J_\alpha (p_\alpha^\bot )\right] , \ \alpha =2,3 .
\end{equation}
Subtracting \eqref{3.19} and \eqref{3.20} from \eqref{3.17} and \eqref{3.18}, respectively, we obtain
\begin{equation*}\label{3.21}
h(X,J_1Y)-h(J_1X,Y)=\frac{1}{2n-1}g(X,J_1Y)J_1(p^\bot_1) ,
\end{equation*}
\begin{equation*}\label{3.22}
h(X,J_\alpha Y)=h(J_\alpha X,Y) , \quad \alpha =2,3 .
\end{equation*}
From the latter two equalities it follows
\begin{equation}\label{3.23}
h(J_1X,J_1Y)+h(X,Y)=\frac{1}{2n-1}g(X,Y)J_1(p^\bot_1) ,
\end{equation}
\begin{equation}\label{3.24}
h(J_\alpha X,J_\alpha Y)=-h(X,Y) , \quad \alpha =2,3 .
\end{equation}
We substitute $J_2X$ and $J_3Y$ for $X$ and $Y$ in \eqref{3.23}, respectively and taking into account the equality \eqref{3.24} we find the first equality in
\eqref{3.8}. The rest equalities in \eqref{3.8} hold because of \eqref{3.3}.
By using \eqref{3.2} we obtain
\begin{equation}\label{3.25}
(\overline \nabla _XJ_\alpha )N=-A_{J_\alpha N}X+J_\alpha (A_NX)+D_XJ_\alpha N-J_\alpha (D_XN) , \quad \alpha =1,2,3 .
\end{equation}
From the conditions \eqref{2.3} and \eqref{2.4} for $\overline F_1$ and $\overline F_\alpha $ ($\alpha =2,3$), respectively, it follows
\begin{equation}\label{3.26}
(\overline \nabla _XJ_1)N=-\frac{1}{2(2n-1)}\left[\overline \theta _1(N)X+\overline \theta _1(J_1N)J_1X\right] ,
\end{equation}
\begin{equation}\label{3.27}
(\overline \nabla _XJ_\alpha )N=-\frac{1}{4n}\left[\overline \theta _\alpha (N)X+\overline \theta _\alpha (J_\alpha N)J_\alpha X\right] , \quad \alpha =2,3 .
\end{equation}
Finally, \eqref{3.25}, \eqref{3.26} and \eqref{3.27} imply \eqref{3.11}, \eqref{3.12}, \eqref{3.13}.
\end{proof}
\begin{theorem}\label{theorem 3.2}
Let $M$ be a holomorphic submanifold of a $\W$-manifold $(\overline M,H,G)$. Then $M$ is a $\W$-manifold (resp. a $\K$-manifold) if and only if all $\overline \theta _\alpha $ ($\alpha =1,2,3$) are non-zero (resp. all $\overline \theta _\alpha $ ($\alpha =1,2,3$) vanish) on $TM$.
\end{theorem}
\begin{proof}
It is clear  that a holomorphic submanifold of an almost hypercomplex manifold with Hermitian and Norden metrics
$(\overline M,H,G)$ supplied with the restrictions of the structures $(H,G)=(J_1,J_2,J_3,g,g_1,g_2,g_3)$ (the restrictions are denoted by the same letters) is also
an almost hypercomplex manifold with Hermitian and Norden metrics. Using \eqref{3.9} and \eqref{3.10}, for the fundamental tensors $F_\alpha $ ($\alpha =1,2,3$) of $M$ we get the following:
\begin{equation}\label{3.28}
\begin{array}{llll}
F_1(X,Y,Z)=\frac{1}{2(2m-1)}\left[g(X,Y)\theta _1(Z)-g(X,Z)\theta _1(Y)\right. \\
\qquad \qquad \qquad\qquad\qquad\left.-g(X,J_1Y)\theta _1(J_1Z)+g(X,J_1Z)\theta _1(J_1Y)\right] ,\\
F_\alpha (X,Y,Z)=\frac{1}{4m}\left[g(X,Y)\theta _\alpha (Z)+g(X,Z)\theta _\alpha(Y)\right. \\
\qquad \qquad \qquad\qquad\left.+g(X,J_\alpha Y)\theta _\alpha (J_\alpha Z)+g(X,J_\alpha Z)\theta _\alpha (J_\alpha Y)\right] ,
\end{array}
\end{equation}
where $X,Y,Z\in TM$ and $\theta _\alpha (Z)=g^{ij}F_\alpha (e_i,e_j,Z)$ ($\alpha =1,2,3$) for an arbitrary basis $\left\{e_i\right\}, \, (i=1, \ldots ,4m)$ of $TM$.
Moreover, the Lee forms $\theta _\alpha $ of $M$ and $\overline \theta _\alpha $ of $\overline M$ are related by the equalities
\begin{equation}\label{3.29}
\theta _1(Z)=\frac{2m-1}{2n-1}\overline \theta _1(Z) , \qquad \theta _\alpha (Z)=\frac{m}{n}\overline \theta _\alpha (Z) , \quad \alpha =2,3 .
\end{equation}
From the latter results it follows that $\theta _1$, $\theta _2$ and $\theta _3$ are either all non-zero, or all vanish.
In the first case, having in mind \eqref{2.3}, \eqref{2.4} and \eqref{3.28}, we obtain that $M$ is a $\W$-manifold. In the case when all $\theta _\alpha $ ($\alpha =1,2,3$)
vanish, from \eqref{3.28} we get that $F_\alpha =0$ ($\alpha =1,2,3$). Therefore, $M$ is a $\K$-manifold.

Conversely, let $M$ be a $\W$-manifold
(resp. a $\K$-manifold). Then all $\theta _\alpha $ ($\alpha =1,2,3$) are non-zero (resp. all $\theta _\alpha $ ($\alpha =1,2,3$) vanish) and the
equalities \eqref{3.29} imply all $\overline \theta _\alpha $ ($\alpha =1,2,3$) are non-zero (resp. all $\overline \theta _\alpha $ ($\alpha =1,2,3$) vanish) on $TM$.
\end{proof}
As an immediate consequence from Theorem \ref{theorem 3.2} and \eqref{3.8} we state
\begin{corollary}\label{corollary 3.1}
Every holomorphic submanifold $M$ of a $\K$-manifold $(\overline M,H,G)$ is a totally geodesic $\K$-manifold.
\end{corollary}
%\begin{proof}
%Since $\overline M$ is a pseudo-hyper-K\"ahler manifold we have that all $\overline \theta _\alpha $ ($\alpha =1,2,3$)
%vanish on $T\overline M$. Hence all $\overline \theta _\alpha $ ($\alpha =1,2,3$) vanish on $TM$, too.
%Then from Theorem \ref{theorem 3.2} it follows that $M$ is a pseudo-hyper-K\"ahler manifold. The condition $\overline \theta _\alpha (Z)=0$, $Z\in T\overline M$ implies $p_\alpha =0$ and according to \eqref{3.4} we get $p_\alpha^\top =p_\alpha^\bot =0$, $\alpha =1,2,3$. Substituting $p_\alpha^\bot =0$ ($\alpha =1,2,3$) in \eqref{3.8} we
%obtain $h(X,Y)=0$ for all $X,Y\in TM$, i.e. $M$ is totally geodesic.
%\end{proof}
\begin{theorem}\label{theorem 3.3}
Let $M$ be a holomorphic submanifold of a $\W$-manifold $(\overline M,H,G)$. Then $M$ is totally umbilical (resp. totally geodesic) if and only if all $\overline \theta _\alpha $ ($\alpha =1,2,3$) are non-zero
(resp. all $\overline \theta _\alpha $ ($\alpha =1,2,3$) vanish) on $TM^\bot $.
\end{theorem}
\begin{proof}
Using \eqref{3.8} we get
\[
\trace h=\frac{2m}{2n-1}J_1(p^\bot_1)=\frac{m}{n}J_\alpha (p_\alpha^\bot ) , \quad \alpha =2,3 .
\]
Hence for the mean curvature vector $C$ we have
\begin{equation}\label{3.30}
C=\frac{1}{2(2n-1)}J_1(p^\bot_1)=\frac{1}{4n}J_\alpha (p_\alpha^\bot ) , \quad \alpha =2,3 .
\end{equation}
According to \eqref{3.8} and \eqref{3.30} we obtain
\begin{equation}\label{3.31}
h(X,Y)=g(X,Y)C, \quad X,Y\in TM.
\end{equation}
Now, from \eqref{3.31} it follows that $M$ is either totally umbilical (when $C\neq 0$), or totally geodesic (when $C=0$). Let $M$ be totally umbilical (resp. totally
geodesic). Then the equalities \eqref{3.30} imply that all $p_\alpha^\bot $ ($\alpha =1,2,3$) are non-zero (resp. all $p_\alpha^\bot $ ($\alpha =1,2,3$) vanish). Hence we obtain that all $\overline \theta _\alpha $ ($\alpha =1,2,3$) are non-zero (resp. all $\overline \theta _\alpha $ ($\alpha =1,2,3$) vanish)
on $TM^\bot $. We prove the sufficient condition of the theorem using \eqref{3.30}.
\end{proof}
Using Theorem \ref{theorem 3.2} and Theorem \ref{theorem 3.3} we obtain the following corollaries.
\begin{corollary}\label{corollary 3.2}
Every totally geodesic holomorphic submanifold $M$ of a $\W$-manifold $(\overline M,H,G)$ is a $\W$-manifold.
\end{corollary}
%\begin{proof}
%From Theorem \ref{theorem 3.3} we have that all $\overline \theta _\alpha $ ($\alpha =1,2,3$) vanish on $TM^\bot $. By using Lemma \ref{lemma 3.3} we obtain that
%all $\overline \theta _\alpha $ ($\alpha =1,2,3$) are non-zero on $TM$. Then from Theorem \ref{theorem 3.2} it follows that $M$ is a $\W$-manifold.
%\end{proof}
\begin{corollary}\label{corollary 3.3}
Every holomorphic $\K$-submanifold $M$ of a $\W$-manifold $(\overline M,H,G)$ is totally umbilical.
\end{corollary}
%\begin{proof}
%From Theorem \ref{theorem 3.2} it follows that all $\overline \theta _\alpha $ ($\alpha =1,2,3$) vanish on $TM$. Then, taking into account Lemma \ref{lemma 3.3}, we
%have that all $\overline \theta _\alpha $ ($\alpha =1,2,3$) are non-zero on $TM^\bot $. By using Theorem \ref{theorem 3.3} we obtain that $M$ is totally umbilical.
%\end{proof}
\begin{corollary}\label{corollary 3.4}
Let $M$ be a holomorphic $\W$-submanifold of a $\W$-manifold $(\overline M,H,G)$.
If all $\overline \theta _\alpha $ ($\alpha =1,2,3$) are non-zero (resp. all $\overline \theta _\alpha $ ($\alpha =1,2,3$) vanish) on $TM^\bot $, then $M$
is totally umbilical (resp. totally geodesic).
\end{corollary}
%\begin{proof}
%From Theorem \ref{theorem 3.2} it follows that all $\overline \theta _\alpha $ ($\alpha =1,2,3$) are non-zero on $TM$. By using Lemma \ref{lemma 3.2} and
%Lemma \ref{lemma 3.3} we conclude that either all $\overline \theta _\alpha $ ($\alpha =1,2,3$) are non-zero on $TM^\bot $, or all $\overline \theta _\alpha $
%($\alpha =1,2,3$) vanish on $TM^\bot $. Applying Theorem \ref{theorem 3.3} we obtain that in the first case $M$ is totally umbilical and in the second case $M$ is
%totally geodesic.
%\end{proof}
\section{Examples of holomorphic submanifolds}
\begin{example}\label{example 4.1}
Let we consider the vector space
\[
{\R}^{4n}=\{p=(x^1,\ldots ,x^n,y^1,\ldots ,y^n,u^1,\ldots ,u^n,v^1,\ldots ,v^n) \;|\; x^i, y^i, u^i, v^i \in {\R}\} .
\]
In \cite{GMD} a hypercomplex structure $H=(J_1,J_2,J_3)$ and a pseudo-Riemannian metric $g$ of signature $(2n,2n)$ are defined on ${\R}^{4n}$ as follows:
\begin{equation}\label{4.1}
\begin{array}{llll}
\displaystyle{
J_1\frac{\partial}{\partial x^i}=\frac{\partial}{\partial y^i}, \quad J_1\frac{\partial}{\partial y^i}=-\frac{\partial}{\partial x^i}, \quad
J_1\frac{\partial}{\partial u^i}=-\frac{\partial}{\partial v^i}, \quad J_1\frac{\partial}{\partial v^i}=\frac{\partial}{\partial u^i};} \cr \cr
\displaystyle{
J_2\frac{\partial}{\partial x^i}=\frac{\partial}{\partial u^i}, \quad J_2\frac{\partial}{\partial y^i}=\frac{\partial}{\partial v^i}, \quad
J_2\frac{\partial}{\partial u^i}=-\frac{\partial}{\partial x^i}, \quad J_2\frac{\partial}{\partial v^i}=-\frac{\partial}{\partial y^i};} \cr \cr
\displaystyle{
J_3\frac{\partial}{\partial x^i}=-\frac{\partial}{\partial v^i}, \quad J_3\frac{\partial}{\partial y^i}=\frac{\partial}{\partial u^i}, \quad
J_3\frac{\partial}{\partial u^i}=-\frac{\partial}{\partial y^i}, \quad J_3\frac{\partial}{\partial v^i}=\frac{\partial}{\partial x^i} ,}
\end{array}
\end{equation}
\begin{equation}\label{4.2}
g(X,X)=\delta _{ij}\left(-p^ip^j-q^iq^j+r^ir^j+s^is^j\right) ,
\end{equation}
where ${X=p^i\frac{\partial}{\partial x^i}+q^i\frac{\partial}{\partial y^i}+r^i\frac{\partial}{\partial u^i}+s^i\frac{\partial}{\partial v^i}, (i=1,2,\ldots ,n})$ is an arbitrary vector field and $\delta _{ij}$ is the Kronecker delta.

From \eqref{4.1} and \eqref{4.2} it follows that the metric $g$ is Hermitian with respect to  $J_1$ and $g$ is a Norden  metric with respect to $J_2$ and $J_3$.
According to \eqref{4.2} the components $g_{ij}$ of the matrix of $g$ with respect to the local basis $\left\{\frac{\partial}{\partial x^i} , \frac{\partial}{\partial y^i} , \frac{\partial}{\partial u^i} , \frac{\partial}{\partial v^i}\right\}$, $\displaystyle{(i=1,2,\ldots ,n)}$ are constants. Therefore the Levi-Civita connection $\overline \nabla $ of the metric $g$
is flat. Moreover, having in mind \eqref{4.1}, it is easy to check that $\overline \nabla J_\alpha =0$, ($\alpha =1,2,3$).  Hence, $({\R}^{4n},H,G)$ is a $\K$-manifold.

Let $M$ be a submanifold of $({\R}^{4n},H,G)$ given by the following immersion:
\begin{equation*}
\begin{array}{lrr}
\varphi (x^1,\ldots ,x^m,y^1,\ldots ,y^m,u^1,\ldots ,u^m,v^1,\ldots ,v^m)\cr
=\left(x^1,\ldots ,x^m,x^{m+1},\ldots ,x^n,y^1,\ldots ,y^m,y^{m+1},\ldots ,y^n,\right. \cr
\phantom{=}\left.\,\,u^1,\ldots ,u^m,u^{m+1},\ldots ,u^n,v^1,\ldots ,v^m,v^{m+1},\ldots ,v^n\right) .
\end{array}
\end{equation*}
Identifying the point $(x^1,\ldots ,x^n,y^1,\ldots ,y^n,u^1,\ldots ,u^n,v^1,\ldots ,v^n)$ in ${\R}^{4n}$ with its position vector $Z$, we obtain that the tangent
bundle $TM$ of $M$ is spanned by
\[
\begin{array}{l}
\left\{\frac{\partial Z}{\partial x^i}=\frac{\partial}{\partial x^i} ,\quad \frac{\partial Z}{\partial y^i}=\frac{\partial}{\partial y^i} ,\quad
\frac{\partial Z}{\partial u^i}=\frac{\partial}{\partial u^i} ,\quad \frac{\partial Z}{\partial v^i}=\frac{\partial}{\partial v^i}\right\} , \quad i=1,2,\ldots ,m .
\end{array}
\]
From \eqref{4.1} and \eqref{4.2} it follows that the complex structures $J_\alpha $ ($\alpha =1,2,3$) preserve $TM$ and the submanifold $M$ is non-degenerate. Thus,
$M$ is a $4m$-dimensional holomorphic submanifold of $({\R}^{4n},H,G)$. The normal bundle $TM^\bot $ is spanned by
\[
\begin{array}{l}
\left\{\frac{\partial}{\partial x^{m+j}} ,\quad  \frac{\partial}{\partial y^{m+j}} ,\quad
\frac{\partial}{\partial u^{m+j}} ,\quad  \frac{\partial}{\partial v^{m+j}}\right\} , \quad j=1,\ldots ,n-m .
\end{array}
\]

Let $X$ be a tangent vector field and $N$ be a normal vector field such that
\[
\begin{array}{l}
N=a^j\frac{\partial}{\partial x^{m+j}}+b^j\frac{\partial}{\partial y^{m+j}}
+c^j\frac{\partial}{\partial u^{m+j}}+d^j\frac{\partial}{\partial v^{m+j}} , \quad j=1,\ldots ,n-m.
\end{array}
\]
Then, taking into account that $\overline \nabla $ is flat, we find for $j=1,\ldots ,n-m$
\[
\begin{array}{l}
\overline \nabla _XN=(Xa^j)\frac{\partial}{\partial x^{m+j}}+(Xb^j)\frac{\partial}{\partial y^{m+j}}
%\phantom{\overline \nabla _XN=}
+(Xc^j)\frac{\partial}{\partial u^{m+j}}+(Xd^j)\frac{\partial}{\partial v^{m+j}}, %\quad
\end{array}
\]
which means that $A_NX=0$. Hence $h(X,Y)=0$, i.e. $M$ is totally geodesic.
\end{example}

\begin{example}
Let $(M,H,G)=(J_1,J_2,J_3,g,g_1,g_2,g_3)$ and $(M^\prime ,H^\prime ,G^\prime )=(J_1^\prime ,J_2^\prime ,\\
J_3^\prime ,g^\prime,g_1^\prime ,g_2^\prime ,g_3^\prime )$ be
almost hypercomplex manifolds with Hermitian and Norden metrics of dimension $4m$ and $4m'$, respectively. Let $\overline M=M\times M^\prime $ be the product manifold
of $M$ and $M^\prime $. Hence, $\overline M$ is a $4(m+m')$-dimensional differentiable manifold and for any vector field $\overline X$ on $\overline M$ the
following decomposition is valid
\[
\overline X=(X,X^\prime )=X+X^\prime  ,
\]
where $X\in TM$ and $X^\prime \in TM^\prime $. Following \cite{K}, we define three almost complex structures $\overline J_1, \overline J_2, \overline J_3$ and a metric
$\overline g$ on $\overline M$ by
\begin{equation}\label{4.3}
\overline J_\alpha (X,X^\prime )=(J_\alpha X,J_\alpha ^\prime X^\prime ) , \quad \alpha =1,2,3 ,
\end{equation}
\begin{equation}\label{4.4}
\overline g((X,X^\prime ),(Y,Y^\prime ))=g(X,Y)+g^\prime (X^\prime ,Y^\prime )
\end{equation}
for arbitrary $(X,X^\prime ), (Y,Y^\prime )\in T\overline M$. According to \eqref{4.4} any vector fields $X\in TM$ and  $X^\prime \in TM^\prime $ are mutually
orthogonal.

It is easy to check that $\overline H=(\overline J_1,\overline J_2,\overline J_3)$ is an almost hypercomplex structure on $\overline M$ and $\overline g$ satisfies \eqref{2.1}. Thus $(\overline M,\overline H,\overline G)$ is an almost hypercomplex manifold with Hermitian and Norden metrics, where by $\overline G$ is
denoted the structure $(\overline g,\overline g_1,\overline g_2,\overline g_3)$. Let $\overline \nabla $, $\nabla $ and $\nabla ^\prime $ be the Levi-Civita connections
of the metrics $\overline g$, $g$ and $g^\prime $, respectively. By $\overline F_\alpha $, $F_\alpha $ and $F_\alpha ^\prime $ $(\alpha =1,2,3)$ we denote the
fundamental tensors on $\overline M$, $M$ and $M^\prime $, respectively, i.e.
\begin{equation*}
\begin{array}{lll}
\overline F_\alpha (\overline X,\overline Y,\overline Z)=\overline g((\overline \nabla _{\overline X}\overline J_\alpha )\overline Y,\overline Z) ; \quad
\alpha =1,2,3 ; \quad \overline X,\overline Y,\overline Z\in T\overline M , \cr
F_\alpha (X,Y,Z)=g((\nabla _XJ_\alpha )Y,Z) ; \quad \alpha =1,2,3 ; \quad X,Y,Z\in TM , \cr
F_\alpha ^\prime (X^\prime ,Y^\prime ,Z^\prime )=g^\prime ((\nabla ^\prime _{X^\prime }J_\alpha ^\prime )Y^\prime ,Z^\prime ) ; \quad \alpha =1,2,3 ;
\quad X^\prime ,Y^\prime ,Z^\prime \in TM^\prime .
\end{array}
\end{equation*}
In \cite{K, GN} it is shown that for $\overline F_\alpha $, $F_\alpha $ and $F_\alpha ^\prime $ $(\alpha =1,2,3)$ and for the corresponding Lee forms
$\overline \theta _\alpha $, $\theta _\alpha $ and $\theta _\alpha ^\prime $ $(\alpha =1,2,3)$ the following interrelations hold
\begin{equation}\label{4.5}
\overline F_\alpha (\overline X,\overline Y,\overline Z)=F_\alpha (X,Y,Z)+F_\alpha ^\prime (X^\prime ,Y^\prime ,Z^\prime ) , \quad \alpha =1,2,3 ,
\end{equation}
\begin{equation}\label{4.6}
\overline \theta _\alpha (\overline Z)=\theta _\alpha (Z)+\theta _\alpha ^\prime (Z^\prime ) , \quad \alpha =1,2,3 ,
\end{equation}
where $\overline X,\overline Y,\overline Z\in T\overline M$, \, $X,Y,Z\in TM$, \, $X^\prime ,Y^\prime ,Z^\prime \in TM^\prime $.

Now, we consider the case when $M$ is a $\K$-manifold and $M^\prime $ is a $\W$-manifold. Then by using \eqref{4.5} we obtain
$\overline F_\alpha (\overline X,\overline Y,\overline Z)=F_\alpha ^\prime (X^\prime ,Y^\prime ,Z^\prime )$, which implies that \\
$\overline F_\alpha (X^\prime ,Y^\prime ,Z^\prime )=F_\alpha ^\prime (X^\prime ,Y^\prime ,Z^\prime )$, $(\alpha =1,2,3)$. Thus we get
\begin{equation}\label{4.7}
\overline F_\alpha (\overline X,\overline Y,\overline Z)=\overline F_\alpha (X^\prime ,Y^\prime ,Z^\prime )=F_\alpha ^\prime (X^\prime ,Y^\prime ,Z^\prime ) ,
\quad \alpha =1,2,3 .
\end{equation}
Taking into account that $M$ is a $\W$-manifold,  \eqref{4.6} and \eqref{4.7} we obtain that $\overline M$ is a $\W$-manifold.

Let $(U\times V,x^1,\ldots ,x^{4m},y^1,
\ldots ,y^{4m'})$, \, $(U,x^1,\ldots ,x^{4m})$ and $(V,y^1,\ldots ,y^{4m'})$ be coordinate neighborhoods on $\overline M$, $M$ and $M^\prime $, respectively. Then $M$
and $M^\prime $ are submanifolds of $\overline M$, given by the natural imbedding maps
\begin{equation*}
\begin{array}{ll}
i_1(x^1,\ldots ,x^{4m})=(x^1,\ldots ,x^{4m},y^1,\ldots ,y^{4m'}) , \cr
i_2(y^1,\ldots ,y^{4m'})=(x^1,\ldots ,x^{4m},y^1,\ldots ,y^{4m'}) ,
\end{array}
\end{equation*}
respectively. From \eqref{4.3} it follows that the submanifolds $M$ and $M^\prime $ are holomorphic. Moreover,
$TM^\bot $ (resp. ${TM^\prime }^\bot $) coincides with $TM^\prime $ (resp. TM).

Since $M$ is a holomorphic $\K$-submanifold of a $\W$-manifold
$\overline M$, from Corollary \ref{corollary 3.3} it follows that $M$ is totally umbilical. Indeed, \eqref{4.6} imply that all
$\overline \theta _\alpha $ $(\alpha =1,2,3)$ are non-zero on $TM^\bot $. Then, according to Theorem \ref{theorem 3.3}, $M$ is totally umbilical.

The holomorphic submanifold $M^\prime $ of a $\W$-manifold $\overline M$ is a $\W$-manifold. Then from \eqref{4.6} we have $\overline \theta _\alpha (Z)=
\theta _\alpha (Z)=0$ $(\alpha =1,2,3)$. The last equalities mean that all $\overline \theta _\alpha $ $(\alpha =1,2,3)$ vanish on ${TM^\prime }^\bot $. Finally, from Corollary \ref{corollary 3.4} it follows that $M^\prime $ is totally geodesic.
\end{example}
\begin{remark}
A series of explicit examples of hypercomplex manifolds with Hermitian and Norden metrics belonging to the classes $\K$ and $\W$ are given in \cite{MS} and  \cite{MM}, respectively.
\end{remark}


\begin{thebibliography}{20}

\bibitem{GMD} Gribachev K., M. Manev, S. Dimiev (2003) On the almost
hy\-per\-complex pseudo-Her\-mit\-ian manifolds. In: Trends in
Compl. Anal., Diff. Geom. and Math. Phys. (Ed. S. Dimiev and K. Sekigawa), World Sci. Publ., 51--62.

\bibitem{MM} Manev M. (2011) A connection with parallel torsion on almost hypercomplex manifolds with Hermitian and anti-Hermitian metrics, J. Geom. Phys., {\bf 61}, 248--259.

\bibitem{MG} Manev M., K. Gribachev (2011) A connection with parallel totally skew-symmetric torsion  on a class of almost hypercomplex manifolds with Hermitian and anti-Hermitian metrics, Int. J. Geom. Methods Mod. Phys., {\bf 8}, No. 1, 115--131.

\bibitem{MS} Manev M., K. Sekigawa (2005) Some four-dimensional almost hypercomplex pseudo-Hermitian manifolds. In: Contemporary Aspects in Compl. Anal., Diff. Geom. and Math. Phys. (Ed. S. Dimiev and K. Sekigawa), World Sci. Publ., 174--186.


\bibitem{AM} Alekseevsky D. V., S. Marchiafava (1996) Quaternionic structures on a manifold and subordinated structures, Ann. Mat. Pura Appl., {\bf CLXXI}, No. IV, 205--273.


\bibitem{GH} Gray A., L. M. Hervella (1980) The sixteen classes of almost Hermitian manifolds and their linear invariants, Ann. Mat. Pura Appl., {\bf CXXIII}, No. IV, 35--58.

\bibitem{GB} Ganchev G., A. Borisov (1986) Note on the almost complex manifolds with a Norden metric, C. R. Acad. Bulgare Sci., {\bf 39}, No. 5, 31--34.

\bibitem{K} Kanemaki S. (1984) On quasi-Sasakian manifolds, Differential geometry, Banach Center Publ., Warsaw, {\bf 12}.

\bibitem{GN} Nakova G. (1999) Four-dimensional submanifolds of seven-dimensional almost contact manifolds with B-metric, Aspects of Compl. Anal., Diff. Geom.,
Math. Phys. and Appl., World Sci. Publ., 148--159.

\end{thebibliography}
\end{document}